\documentclass[11pt]{article}

\usepackage[margin=1in]{geometry}
\usepackage{amsmath,amsthm,amssymb}

\usepackage{todonotes}

\usepackage{tikz}
\usepackage{tkz-graph}

\usepackage[hidelinks]{hyperref}
\usepackage{mathtools}
\usepackage{mathabx}

\usepackage[headings]{fullpage}
\usepackage{amssymb,amsthm,amsfonts,amsmath}
\usepackage{enumitem}
\usepackage{hyperref}
\usepackage{mathrsfs}
\usepackage{graphicx}
\usepackage{mdframed}
\usepackage{bm}
\usepackage{float}

\usepackage{multicol}
\usepackage{svg}

\usepackage{listings}

\definecolor{codegreen}{rgb}{0,0.6,0}
\definecolor{codegray}{rgb}{0.5,0.5,0.5}
\definecolor{codepurple}{rgb}{0.58,0,0.82}
\definecolor{backcolour}{rgb}{0.95,0.95,0.92}
 
\lstdefinestyle{mystyle}{
    backgroundcolor=\color{backcolour},   
    commentstyle=\color{codegreen},
    keywordstyle=\color{violet},
    numberstyle=\tiny\color{codegray},
    stringstyle=\color{codepurple},
    basicstyle=\ttfamily\footnotesize,
    breakatwhitespace=false,         
    breaklines=true,                 
    captionpos=b,                    
    keepspaces=true,                 
    numbers=left,                    
    numbersep=5pt,                  
    showspaces=false,                
    showstringspaces=false,
    showtabs=false,                  
    tabsize=2
}
 
\usepackage{CJKutf8}

\hypersetup{
colorlinks,
citecolor=black,
filecolor=black,
linkcolor=black,
urlcolor=black
}

\setlist{nolistsep}
\newcommand{\M}{\mathcal{M}}

\newtheoremstyle{break}
    {\topsep}{\topsep}%
    {\upshape}{}%
    {\bfseries}{}%
    {\newline}{}%
\theoremstyle{break}

\theoremstyle{plain}
\newtheorem{theorem}{Theorem}[section]
\newtheorem{proposition}[theorem]{Proposition}

\newtheorem{lemma}[theorem]{Lemma}

\newtheorem{definition}[theorem]{Definition}
\newtheorem{example}[theorem]{Example}

\newtheorem*{thm:main}{Theorem \ref{thm:main}}

\newtheorem*{thm:simple-basis}{Theorem  \ref{prop:simple-basis}}
\newtheorem*{thm:cycle-basis-comp}{Theorem \ref{thm:cycle-basis-comp}}
\newtheorem*{thm:basis-extension}{Theorem \ref{thm:basis-extension}}
\newtheorem*{thm:p-cycle-space}{Theorem \ref{thm:p-cycle-space}}
\newtheorem*{thm:binary-interpolation}{Theorem \ref{thm:binary-interpolation}}
\newcommand{\BB}{\mathcal{B}}

\graphicspath{{images/}}

\begin{document}

% \title{Bases of the Rational and $\F_p$-Cycle Spaces of Undirected Graphs}
% \title{Bases of the $K$-Cycle Spaces of Undirected Graphs}
% \title{Algorithmic Methods for Cycle Space Bases \\ of an Undirected Graph over Non-binary Fields}
\title{Catalan Recursion on Externally Ordered Bases of \\Unit Interval Positroids}

\author{Jan Tracy Camacho}

\maketitle

% n = number of nodes
% k = value for k-connected
% m = number of edges
% r = length of Tutte decomposition

\section{Introduction}

The Catalan numbers have a rich history. Its ubiquity in mathematics underscores the importance of this sequence. The Catalan sequence has appeared time and again in the works of numerous mathematicians. There are so many objects counted by the Catalan numbers, there's a whole book dedicated to exploring them \cite{stanCat}. Some of those combinatorial objects are triangulations of convex polygons with $n+2$ vertices (Fig. \ref{fig:triangulation}), binary trees with $n+2$ vertices, plane trees with $n+1$ vertices, ballot sequences of length $2n$, parenthetizations, and Dyck paths of length $2n$.

\begin{figure}[ht]
    \centering
    \includegraphics[width=6cm]{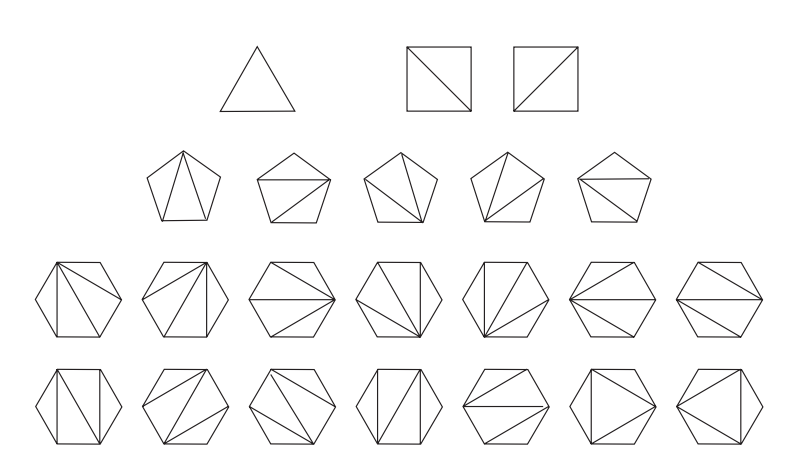}
    \caption{Triangulations of convex polygons with 3, 4, 5, and 6 vertices}
    \label{fig:triangulation}
\end{figure}

The modern way we describe Catalan numbers stems from Eug\`ene Catalan's interest in the triangulation of polygons problem \cite{stanCat}. Catalan numbers are defined as follows,

\begin{equation*} \label{CatChoose}
    C_n = \frac{1}{n+1} \binom{2n}{n}
\end{equation*}

We can also define Catalan numbers in terms of its generating function,
\begin{equation*}\label{CatGenFunc}
    C(x) = \sum_{n\geq 0} = 1+x+2x^2+5x^3+14x^4+\dots,
\end{equation*}
where the $n^{th}$ coefficient is the $n^{th}$ Catalan number. 

Euler and Goldbach studied the triangulation of convex polygons which gives us the following recursive definition,
\begin{equation*} \label{FundRec}
C_{n+1} = \sum_{k=0}^{n} C_n C_{n-k},\hspace{1em} C_0=1.
\end{equation*}

Many objects enumerated by the Catalan numbers exhibit the same recursive property of the numbers themselves. That is, given a Catalan object, a rule exists that describes how to produce an object from the objects that came before it.

We introduce a new Catalan object called the \emph{externally ordered poset of unit interval positroid} and explore its potential recursion. The elements of this poset are bases of a special matroid called a Unit Interval Positroid (UIP), a Catalan object introduced by Chavez--Gotti \cite{chavez}. The partial order $\leq_{\mbox{Ext}}$ is the external order of matroid bases as defined by Las Vergnas \cite{lvergnas}. We denote such a poset as $\mbox{Ex}(\mathcal{P})$, where $\mathcal{P}$ is a positroid. Our main result is the following theorem:
\begin{thm:main} Let $\mathcal{P}_n = ([2n],\BB)$ be the trivial UIP and $\mbox{Ex}(\mathcal{P}_n)$ be the poset of externally ordered bases of $\mathcal{P}_n$. Let $\gamma$ be the algorithm defined in Section \ref{algorithm}. Then $\gamma(\mbox{Ex}(\mathcal{P}_n))=\mbox{Ex}(\mathcal{P}_{n+1})$, where the partial order induced by $\gamma(\mbox{Ex}(\mathcal{P}_n))$ is the external order.
\end{thm:main}
Our main result shows recursion can be described nicely for the poset of the trivial UIP as well as enumerate the bases of the positroid. This is the first step towards describing the complete recursion for these special posets.

This paper is organized as follows. In section 2 we provide the background material necessary for this paper. In Section 3 we introduce an algorithm to generate the externally ordered bases of a rank $n+1$ trivial UIP from a rank $n$ trivial UIP. In section 4 we present our main theorem which proves the algorithm preserves ordering. In section 5 we discuss future work.

\section{Background} \label{background}

%% MATROIDS %%
\subsection{Matroids}

Matroids capture the essence of dependence as we know it in linear algebra and graph theory. Though several definitions of a matroid exist, we utilize the basis definition in our work. 

\begin{definition}A \emph{matroid} $\M$ is a pair $(E,\BB)$ consisting of a finite set $E$ and a nonempty collection of subsets $\BB=\BB(\M)$ of $E$ that satisfy the following properties:
	\begin{enumerate}
		\item[(B1)] $\BB\neq\emptyset$
		\item[(B2)] (Basis exchange axiom) If $B_1,B_2\in\BB$ and $b_1\in B_1-B_2$, then there exists an element $b_2\in B_2-B_1$ such that $B_1-\{b_1\}\cup\{b_2\}\in\BB$.
	\end{enumerate} 
	The element $B\in \BB$ is called a \emph{basis} and $|B|$ is the rank of $\M$. A minimally dependent set $C$, that is $C-e$ is independent for any $e\in C$, is called a \emph{circuit}. The finite set $E$ is called the \emph{ground set.}
\end{definition}

\begin{example}\label{ex:p2_positroid}
Let $E=\{1,2,3,4\}$ be the columns of the following matrix,
\vspace{-1em}
\begin{center}
         \[A=
    \bordermatrix{ & \phantom{-}1 & 2 & \phantom{-}3 & \phantom{-}4 \cr
       & \phantom{-}1 & 0 & -1 & -1 & \cr
       & \phantom{-}0 & 1 & \phantom{-}1 & \phantom{-}1 }.
    \]
    \end{center}
Then the set of bases $\BB$ of the matroid over $E$ is the collection of all maximally independent sets of $E$. That is, $\BB = \{12,13,14,23,24\}$. Considering all the sets of minimally dependent sets of the columns of $A$, we get that the set of circuits of $\M$ are $\mathcal{C} = \{123,124,34\}$. One can easily see that $\BB$ satisfies the bases axioms above.
\end{example}

The matroid described in Example \ref{ex:p2_positroid} is part of a larger class of matroids called realizable. A matroid $\M$ whose bases are in bijection with the sets of maximally independent columns of some matrix $A$ over some field $\mathbb{F}$ is called \emph{realizable} and denoted $\M(A)$.

A matrix $A$ is called \emph{totally nonnegative} if all of its maximal minors are nonnegative. Thus, one can impose this restriciton to the family of realizable matroids and arrive at the following definition due to Postnikov \cite{pos}.

\begin{definition}
A \emph{positroid} of rank $r$ over $[n]$ is a realizable matroid such that the associated full rank $r\times n$ matrix $A$ is totally nonnegative. 
\end{definition}

The matroid $\M(A)$ given in example \ref{ex:p2_positroid} is a positroid of rank 2. In fact, $\M(A)$ is a unit interval positroid, which we discuss in the next section.

%% POSETS and UNIT INTERVAL POSITROIDS %%
\subsection{Unit Interval Positroids}

Recall that a \emph{partially ordered} set, $P$ or \emph{poset}, is a set with a relation $\leq$ that satisfies reflexivity, antisymmetry, and transitivity \cite{stan}. We can represent the poset $(P,\leq)$ as a Hasse diagram that shows the elements of $P$ with the cover relations. See figure \ref{fig:UIO rank 6}.

\begin{definition}
A \emph{poset} $P$ is a \emph{unit interval order} if there exists a bijective map $i \mapsto [q_i,q_{i}+1]$ from $P$ to a set $S=\{ [q_i,q_{i}+1] | 1\leq i \leq n, q_i \in \mathbb{R} \} $ of closed unit intervals of the real line such that for $i,j\in P, i <_P j$ if and only if $q_i+1 < q_j$. We then say that $S$ is an \emph{interval representation} of $P$.
\end{definition}

\begin{figure}[h]
    \centering
    \includegraphics[width=0.75\linewidth]{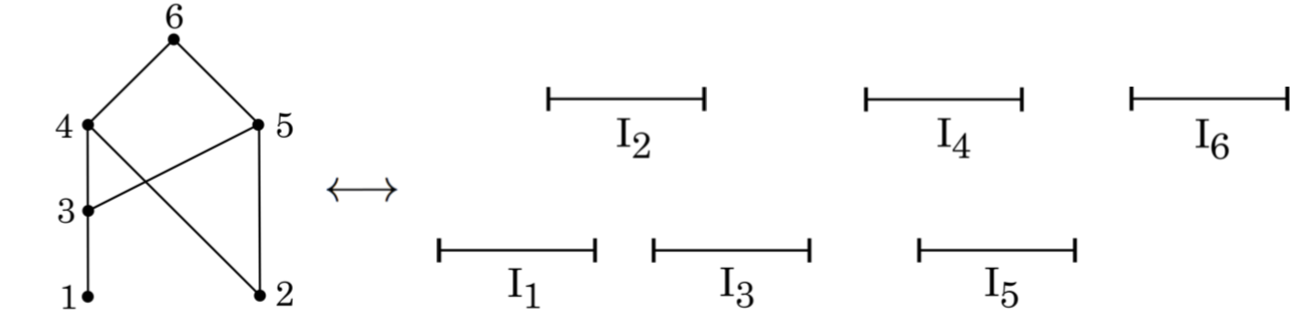}
    \caption{Hasse diagram of a unit interval order with 6 elements and its unit interval representation.}
    \label{fig:UIO rank 6}
\end{figure}

Associated to every unit interval order is an antiadjacency matrix, the key to describing unit interval positroids.

\begin{definition}
For an $n$-labeled poset $P$, the \emph{antiadjacency matrix} of $P$ is the $n \times n$ binary matrix $A=(a_{i,j})$ with $a_{i,j} =0$ if and only if $i <_P j$.
\end{definition}

Skandera--Reed \cite{skandera} proved that by labeling the unit interval order appropriately, every minor of the corresponding antiadjacency matrix $A$ is nonnegative. That is, the determinant of every submatrix of $A$ is 0 or a positive number. This fact, combined with the following result, leads us to our positroid of interest.

\begin{lemma}\label{lem:induced_positroid}(Postnikov, \cite{pos} Lem 3.9) % Lem 2.5 Chavez & Gotti (from Postnikov)

For an $n \times n$ real matrix $A=(a_{i,j})$, consider the 
$n \times 2n$ matrix $B = \psi(A)$ where
\small{
\begin{equation*}
\arraycolsep=3pt
\left( \begin{array}{ccc}
a_{1,1}& \cdots & a_{1,n}\\
\vdots& \ddots& \vdots\\
a_{n-1,1}& \cdots& a_{n-1,n}\\
a_{n,1}& \cdots& a_{n,n}
\end{array} \right)
\overset{\psi}{\longmapsto}
\left( \begin{array}{cccccccc}
1&\cdots&0&0&(-1)^{n-1}a_{n,1}& \cdots &(-1)^{n-1}a_{n,n} \\
\vdots&\ddots&\vdots&\vdots& \vdots & \ddots & \vdots\\
0 & \cdots & 1 & 0  & -a_{2,1} &\cdots & -a_{2,n}\\ 
0&\cdots&0 &1 & \phantom{1}a_{1,1} & \cdots & \phantom{1}a_{1,n}
\end{array}\right).
\end{equation*}
For each pair $(I,J)$ with $I,J \subseteq [n]$ and $|I|=|J|$, define the set 
\begin{equation*}
    K = K(I,J) = \{n+1-k|k\in [n]\backslash I\}\cup \{n+j|j\in J\}.
\end{equation*}}
Then we have $\Delta_{I,J}(A)=\Delta_{[n],K}(B)$.
\end{lemma}

Lemma \ref{lem:induced_positroid} gives us a way to encode the same information from the $n\times n$ antiadjacency matrix into a corresponding $n\times 2n$ matrix. Explicitly, Lemma \ref{lem:induced_positroid} associates the determinants of the submatrices of the antiadjacency matrix to the maximal minors of the $n\times 2n$ matrix. The antiadjacency matrix of a properly labeled unit interval order generates a totally nonnegative matrix, that is, a positroid. A positroid on $[2n]$ induced by a unit interval order is a \emph{unit interval positroid} or UIP.

Suppose we have a unit interval positroid $\mathcal{P}$ and we let $C$ be a circuit such that $C\in \mathcal{C(P)}$. If we take an element $e\in C$ from circuit $C\in \mathcal{C}(P)$ and $e$ is the smallest, then $C - \{e\}$ is a \emph{broken circuit}. Bases that do not contain broken circuits are called \emph{atoms}. However, if $B$ is a broken circuit, then there is an element $e$ such that $B \cup \{e\}$ is a circuit and $e$ is the smallest. This notion of broken circuits is captured by the ordering discussed in the next section.

%% EXTERNAL ORDER ON  BASES and Externally Ordered Posets %%
\subsection{External Activities}
In \cite{lvergnas} Las Vergnas introduced the notion of active orders, a collection of related orders on bases of a matroid using broken circuits. We use only the external order to define our poset of interest, though one can derive the results for other orders from this one.

\begin{definition}
Let $M$ be a matroid on a linearly ordered set $E$, and let $A\subseteq E$. 
We say an element $e\in E$ is \emph{M-active} with respect to $A$ if there is a circuit $C$ of $M$ such that $e\in C \subseteq A \cup \{e\}$ and $e$ is the smallest element of $C$. We denote by $\mbox{Act}_{M}(A)$ the set of \emph{M-active} elements with respect to $A$.
\end{definition}

To determine the active elements for our positroids we look at those subsets $A\subseteq E$ where $A$ is a basis. Then those elements $e$ are exactly the elements which make or break a circuit.

\begin{definition}
The \emph{external set} of an element $A$ is obtained by setting $\mbox{Ext}_{M}(A)=\mbox{Act}_{M}(A)\backslash A$.
\end{definition}

The following proposition defines the external order of a set of bases of a matroid.

\begin{proposition}\label{lem:lasvergnas_def} (\cite{lvergnas}, Proposition 3.1)
Let $A$, $B$ be two bases of an ordered matroid $M$. The following properties are equivalent:
\begin{enumerate}
    \item $A \leq_{\mbox{Ext}} B$;
    \item $A \subseteq B \cup \mbox{Ext}_{M}(B)$;
    \item $A \cup \mbox{Ext}_{M}(A) \subseteq B \cup \mbox{Ext}_{M}(B)$
    % Do we need this next one? 
    \item B is the greatest, for the lexicographic ordering, of all bases of $M$ contained in $A \cup B$.
\end{enumerate}

\end{proposition}

Let $\mathcal{P}$ be a unit interval positroid generated by the poset where all elements are incomparable. Then we call $\mathcal{P}$ \emph{trivial}. We utilize Proposition \ref{lem:lasvergnas_def}(2) on the trivial UIP in order to determine the externally ordered poset of its bases. For an example of generating the matrix representing the trivial UIP of rank 3, see Figure \ref{fig:trivialUIP3}.

\begin{figure}
    \centering
    \includegraphics[width=0.85\linewidth]{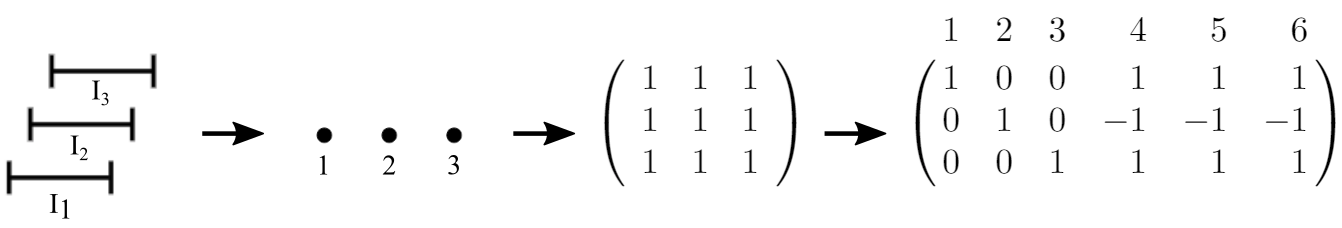}
    \caption{A trivial UIP of rank 3}
    \label{fig:trivialUIP3}
\end{figure}
\pagebreak
\begin{example}\label{Ex: calculating external}
Consider the trivial UIP of rank 3, $\mathcal{P}_3$. We will determine the external ordering on its bases. Using proposition \ref{lem:lasvergnas_def}, we first determine the active elements for every basis $B\in\BB$, that is $\mbox{Act}_{\mathcal{P}_3}(\BB)$. The bases and circuits of $\mathcal{P}_3$ are described as follows,

\begin{align*}
\mathcal{C}_{\mathcal{P}_{3}}  = \{&(1,2,3,4),(1,2,3,5),(1,2,3,6),(4,5),(4,6),(5,6)\},\\
\BB_{\mathcal{P}_{3}}  = \{&(1,2,3),(1,2,4),(1,2,5),(1,2,6),(1,3,4),(1,3,5),(1,3,6),\\   
    & (2,3,4),(2,3,5),(2,3,6)\}.
\end{align*}

The following tables show necessary calculations.
\begin{center}
\begin{multicols}{2}
% ACTIVE ELEMENT Table
\begin{tabular}{ |p{0.5cm}|p{1.7cm}|p{2.3cm}|  }
\hline
\multicolumn{3}{|c|}{Calculating M-Active Elements} \\
\hline
$e$& $c\in\mathcal{C_{P}}$&$A\cup \{e\}$ \\
\hline
1 & $(1,2,3,4)$ & $(2,3,4) \cup \{1\}$ \\
 & $(1,2,3,5)$ & $(2,3,5) \cup \{1\}$ \\
 & $(1,2,3,6)$& $(2,3,6) \cup \{1\}$ \\
2 & none  & \\
3 & none &  \\
4 & $(4,5)$ & $(1,2,5)\cup \{4\}$ \\
 & &$(1,3,5)\cup \{4\}$ \\
 & &$(2,3,5)\cup \{4\}$ \\
 & $(4,6)$ &$(1,2,6)\cup \{4\}$ \\
 & &$(1,3,6)\cup \{4\}$ \\
 & &$(2,3,6)\cup \{4\}$ \\
5 & $(5,6)$& $(1,2,6)\cup \{5\}$  \\
 & & $(1,3,6)\cup \{5\}$  \\
 & & $(2,3,6)\cup \{5\}$  \\
6 & none & \\
\hline
\end{tabular}

% EXTERNAL SET TABLE
\begin{tabular}{ |p{1.4cm}|p{1.5cm}|p{1.5cm}|  }
\hline
\multicolumn{3}{|c|}{Calculating the External Set} \\
\hline
$A$& $\mbox{Act}_{M}(A)$&$\mbox{Ext}_{M}(A)$ \\
\hline
$(1,2,3)$ & none & $\emptyset$ \\
$(1,2,4)$ & none & $\emptyset$\\
$(1,2,5)$ & 4   & 4 \\
$(1,2,6)$ & 4,5 & 4,5 \\
$(1,3,4)$ & none & $\emptyset$ \\
$(1,3,5)$ & 4 & 4 \\
$(1,3,6)$ & 4,5 & 4,5\\
$(2,3,4)$ & 1   & 1 \\
$(2,3,5)$ & 1,4 & 1,4 \\
$(2,3,6)$ & 1,4,5 & 1,4,5 \\
\hline
\end{tabular}
\end{multicols}
\end{center}

To form the externally ordered poset, we use Proposition \ref{lem:lasvergnas_def}(2) in order to obtain the relations described by \ref{lem:lasvergnas_def}(1). For example, for bases $A=(1,2,3)$ and $B=(2,3,4)$, we check if the containment $A \subseteq B \cup \mbox{Ext}_{M}(B)$ is satisfied. We can observe that $(1,2,3)\subseteq (2,3,4) \cup \{1\}$, and thus $A \leq_{\mbox{Ext}} B$. Continue this calculation for all pairs of bases until all bases are ordered. The resulting object, which is a set of bases ordered by the relation $\leq_{\mbox{Ext}}$, is the externally ordered poset of the bases of $\mathcal{P}_3$.
\end{example}
Recall, we denote the poset of the trivial UIP bases with the external order as $\mbox{Ex}(\mathcal{P}_n)$. Continuing with Example \ref{Ex: calculating external}, we will show how to compute $\mbox{Ex}(\mathcal{P}_2)$. We compare the bases of $\mathcal{P}_2$ pairwise using Proposition \ref{lem:lasvergnas_def}(2). That is, for every $A, B\in \BB_{\mathcal{P}_2}$, check if $A \subseteq B \cup \mbox{Ext}_{M}(B)$ is satisfied. Then we can draw the following poset (figure on the right) given the information from the table of $\mbox{Ex}(\mathcal{P}_2)$. Using the information in the table below, and the condition just mentioned, we produce the Hasse diagram for $\mbox{Ex}(\mathcal{P}_2)$.

\begin{figure}[h]
    \centering
    \includegraphics[width=0.50\linewidth]{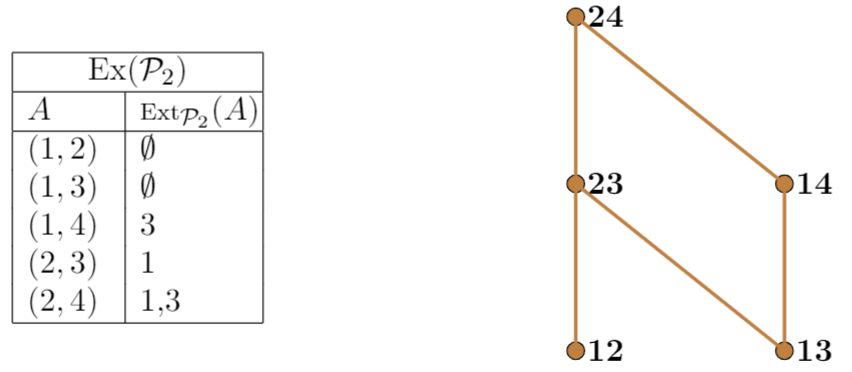}
    \caption{$\mbox{Ex}(\mathcal{P}_2)$ and Hasse diagram.}
    \label{fig:p2 and poset}
\end{figure}

Comparing every pair of bases to find the ordering is easy whenever the rank of the positroid is low. Quickly, this becomes a difficult task for higher ranks. We now introduce our algorithm which streamlines this process.

\section{Algorithm}\label{algorithm}

To prove $\mbox{Ex}(\mathcal{P})$ exhibits Catalan recursion we first describe an algorithm that generates a set of elements from the set of bases of $\mathcal{P}_n$ and defines an order induced by $\mbox{Ex}(\mathcal{P}_n)$. In the next section we prove this set is precisely the set of bases of $\mathcal{P}_{n+1}$ and the induced order preserves the external order of the bases of $\mathcal{P}_{n+1}$. 

Let $\BB_{\mathcal{P}_n}$ be the bases of the trivial UIP $\mathcal{P}_n$. Then we produce a set of elements in the following way. 

\begin{enumerate}%\compresslist
\item \textcolor{blue}{``Reinforce"}: for $B\in \BB_{\mathcal{P}_n}$, add 1 to all elements of $B$ and adjoin $\min{B}$. The original order is kept.
\item \textcolor{red}{``Build Up"}: if $2n+1\in B$ from step 1, then form $B\backslash\{2n+1\}\cup\{2(n+1)\}$ and it covers $B$.
\item \textcolor{olive}{``Grow Spine"}: if $2\in B$ from steps 1 and 2, then form $B\backslash\{2\}\cup\{1\}$ and is covered by $B$.

\end{enumerate}

This completes the procedure for the algorithm. Figure \ref{fig:P2_to_p3} illustrates how the algorithm is applied to the external poset of the trivial UIP of rank 2. 

\begin{figure}[h]
    \centering
    \includegraphics[width=1\linewidth]{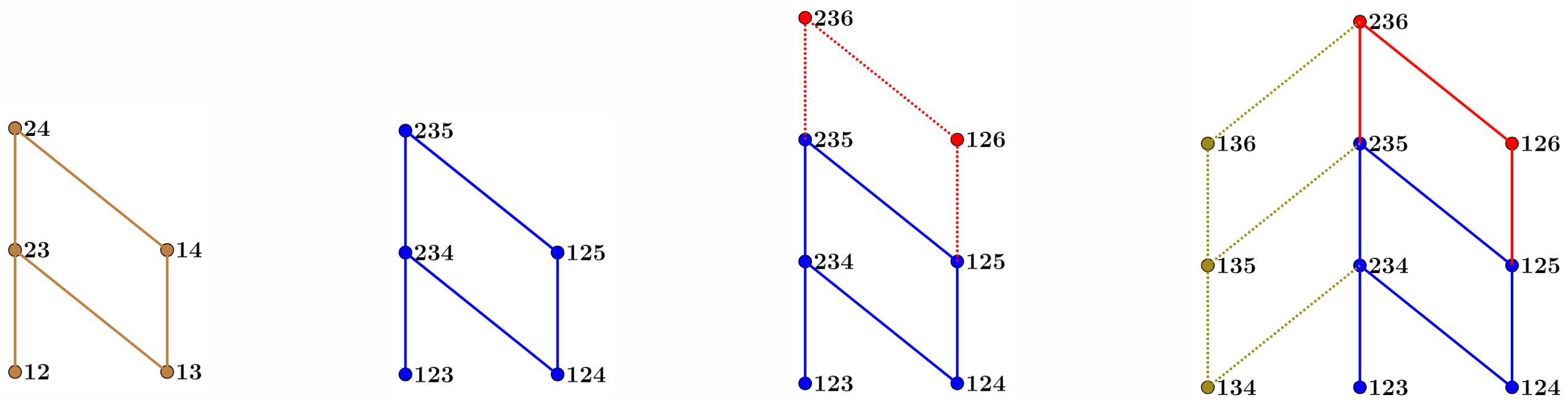}
    \caption{The algorithm applied to the trivial UIP $\mathcal{P}_2$ and the  resulting ordered set.}
    \label{fig:P2_to_p3}
\end{figure}

\section{Main Theorem}\label{}

Denote the algorithm above by $\gamma$; that is, $\gamma(\BB_{\mathcal{P}_n})$ is the ordered set of elements produced by $\gamma$. To prove our main result, we show $\gamma(\BB_{\mathcal{P}_n})$ is precisely the externally ordered poset of the bases of the rank $n+1$ trivial UIP. That is, we show $\gamma$ produces the set of bases expected and the bases are externally ordered. 
We begin by proving the bases and circuits of the trivial UIP have a nice description.

\begin{lemma}\label{lem:bases_description} 

Let $\mathcal{P}$ be the trivial unit interval positroid of rank $n$ and $B\in \BB$ a basis of $\mathcal{P}$. Then $B$ satisfies one of the following

\begin{enumerate}
    \item $B = (1,\hdots,n)$,
    \item $B = (1,\hdots,\hat{j},\hdots,n)\bigcup \{j\}$ for $j\in [n+1,\hdots,2n]$,
\end{enumerate}
where $\hat{j}$ indicates $j$ is not included in the basis $B$. Moreover, $|\BB| = n^2+1.$
\end{lemma}

\begin{proof}
By Lemma \ref{lem:induced_positroid}, we know that $P$ is represented by the $n \times 2n$ matrix,

\[\begin{bmatrix}
1&0&\cdots &0&(-1)^{n-1}& \cdots &(-1)^{n-1} \\
0& 1 &\cdots &0 & \vdots & \ddots & \vdots\\
\vdots & \ddots & 1 & 0  & (-1) &\cdots & (-1)\\
0&\cdots&0 &1 & 1 & \cdots & 1\\
\end{bmatrix},\]\\
where the first $n$ columns form the $n\times n$ identity matrix and the last $[n+1,\hdots, 2n]$ columns are equal. 

Since a basis corresponds to a non-zero $n\times n$ maximal minor, we need only describe which submatrices are non-singular. The $n\times n$ identity matrix is non-singular, which is the first basis description.

Since the $[n+1,\hdots,2n]$ columns are all the same, any non-singular maximal submatrix can include at most one of these columns. Thus, all other bases are formed by maximal submatrices with one column from $[n+1,\hdots,2n]$ and the remaining a subset of the columns $[n]$, which is the second description.

We use our description to enumerate the bases. To form a basis of the second kind, we remove an element from $\{1,2,\dots,n\}$ and replace it with any of the $n$ elements from $[n+1,\dots,2n]$. This is done in $\binom{n}{1}n$ ways. Including the basis $\{1,2,\dots,n\}$, we get $$\binom{n}{1} n + 1 =n^2 + 1$$ bases, as desired.
\end{proof}

\begin{lemma}\label{circuit_description}
Let $\mathcal{P}_n$ be the trivial UIP and $\mathcal{C}_{\mathcal{P}_n}$ its set of circuits. Then any $C\in \mathcal{C}_{\mathcal{P}_n}$ must either be $I\cup i$ where $i\in [n+1,2n]$ or $(i,j)$ for $i,j\in[n+1,2n]$ and $i\neq j$.
\end{lemma}

\begin{proof}
Note that $I\cup i$ where $i\in [n+1,2n]$ is a circuit since the removal of any element produces a basis as described in Lemma \ref{lem:bases_description}. Similarly, the pair $i,j\in[n+1,2n]$ and $i\neq j$ is a circuit since every singleton $i\in[2n]$ is independent. Since any other subset of $[2n]$ contains one of these sets, there are no other circuits of $\mathcal{P}_n$.
\end{proof}

We state the following observation about bases of trivial UIPs, though it is not used in any subsequent theorem. 
\begin{lemma}\label{min_elts_description}
Let $B\in \mathcal{B}(\mathcal{P}_n)$, where $\mathcal{P}_{n}$ is a trivial unit interval positroid. Then $\epsilon(B)=0$ if and only if $B$ is a minimally ordered basis.
\end{lemma}

\begin{proof}
$(\Rightarrow)$ Assume $\epsilon(B)=0$ (i.e. Ext$(B)=\emptyset$) and there exists $A\in \mathcal{B}$ such that $A\lessdot B$. Then $A\subset B\cup \text{Ext}(B) = B$, which implies $A=B$.

$(\Leftarrow)$ Assume $B$ is minimal and $\text{Ext}(B)\neq\emptyset.$ First, note $B$ minimal implies $1\in B$. If not, let $B'=B\backslash2\cup 1$. Then Ext$(B')=$ Ext$(B)\backslash\{1\}$, so that $B'\subset B \cup $Ext$(B)$, implying $B'\lessdot B$, a contradiction. Since $\text{Ext}(B)\neq\emptyset,$ there exists $e=\min(C)$ for some circuit $C$ of $P_n$ such that $C\subset B\cup\{e\}$. By Lemma \ref{circuit_description}, $C=ej$ for some $j\in[n+1,2n]$ and $j\neq e$. Then Ext$(B)=\{n+1,\dots,e\}$, and $j\in B$ necessarily. Let $A=B\backslash \{j\} \cup \{e\}$. Then $A\lessdot B$, a contradiction.
\end{proof}

Next we prove $\gamma$ produces the set of desired bases.

\begin{theorem}\label{thm:AlgMakesBases}
Let $\BB$ be the set of bases of the trivial unit interval positroid $\mathcal{P}_{n}$ of rank $n$ over ground set $[2n]$. Then, after applying algorithm $\gamma$ to $\BB$, $\gamma(\BB)$ is the set of bases for the trivial unit interval positroid $\mathcal{P}_{n+1}$ of rank $n+1$ over ground set $[2(n+1)]$.
\end{theorem}

\begin{proof}
Let $\mathcal{B}=\mathcal{B}(\mathcal{P}_{2n})$. We will show that for every $B\in\mathcal{B}$, $\gamma(\mathcal{B})$ produces a $B^{'}\in\mathcal{B}(\mathcal{P}_{2(n+1)})$.
 Let $B=\{1,\dots,n\}.$ Then $\gamma(B)=\{1,\dots,n+1\}$, which by Lemma \ref{lem:bases_description}(1) is a basis of $\mathcal{P}_{2(n+1)}$.
 Let $B=\{1,2,\dots,\hat{j},\dots,n\}$, where $\hat{j}\in [n+1,\dots, 2n]$. Applying step 1 of $\gamma$ we have that $\gamma(B) = B+1\cup\{min(B)\}\in\gamma(\BB).$ Explicitly, this means $$\{1,2,3,\dots, \hat{j}+1,\dots,n+1\}\in\gamma(\BB),$$ where $$\hat{j}+1 \in [n+2,2n+1].$$ Since $[n+2,\dots,2n+1]\subset [(n+1)+1,\dots,2(n+1)]$, then by Lemma \ref{lem:bases_description} we have that $\gamma(B)$ is a basis of $\mathcal{P}_{2(n+1)}$. Moreover, there are $n^2$ bases created at this step.
 
For the second step of $\gamma$, let $\BB_{max}$ be the set of bases such that max$(\gamma(B))=2n+1$. Then for all $B\in \BB_{max}$, we have $B\backslash\{2n+1\}\cup \{2(n+1)\}$. All sets generated prior to this step have maximum value $2n+1$. Thus, all sets generated at this step of $\gamma$ are new elements of $\gamma({\BB})$. Every new set generated is formed by choosing a value of $[n]$ and replacing it with $2(n+1)$. Thus each set is a basis and we have $n$ new bases.
 
For the final step of $\gamma$, let $\BB_{min}$ be the set of bases such that min$(\gamma(B))=2$. Then for every $B\in \BB_{min}$, $B\backslash\{2\}\cup \{1\}\in\gamma(\BB)$. Note Lemma \ref{lem:bases_description} implies $3\in B$ for every $B\in \BB_{min}$. Then for every set generated at this step, it must contain $1,3$ and not $2$. By construction, all sets generated by $\gamma$ prior to the third step of $\gamma$ have that their first two entries are consecutive. Thus, all sets generated in this third step are new and, by Lemma \ref{lem:bases_description}, a basis. Thus all sets are bases of $\mathcal{P}_{n+1}$.

To enumerate the bases generated in this final step of $\gamma$, let us analyze the number of bases it applies to. By the first step of the algorithm, every basis $B\in\mathcal{B}$ such that $\min(B)=2$ produces a basis in $\mathcal{B}(\mathcal{P}_{2(n+1)})$ with minimum element 2. Thus, the number of bases produced equals the number of $B\in\mathcal{B}$ such that $\min(B)=2$, which is $n$. By the second step of the algorithm, the only basis produced with minimum element 2 is the one that also has maximum element $2n$. Thus in total, there are $n+1$ bases with minimum element 2 produced by $\gamma$. Accounting for the all the bases produced at every step of $\gamma$, we have that $$|\gamma(\mathcal{B}(\mathcal{P}_{2(n+1)}))|=n^2+n+n+1+1 = (n+1)^2+1,$$ as desired.
\end{proof}

Using the above results, we may now prove our main theorem.

\begin{theorem}\label{thm:main}
Let $\mathcal{B}(\mathcal{P}_n)$ be the bases of the rank $n$ trivial UIP on $[2n]$. Assume $\gamma$ is defined to be the algorithm described in Section \ref{algorithm}. Then $\gamma(\mathcal{B}(\mathcal{P}_n))=\mbox{Ex}(\mathcal{P}_{n+1})$, the externally ordered poset of bases of the rank $n+1$ trivial UIP on $[2(n+1)]$.
\end{theorem}

\begin{proof}
By Lemma \ref{thm:AlgMakesBases}, we know $\gamma(\mathcal{B}(\mathcal{P}_n))$ produces the set of bases of $\mathcal{P}_{n+1}.$ It is left to show that $\gamma$ also induces the external order on $\mathcal{B}(\mathcal{P}_{n+1}),$ thus producing $\mbox{Ex}(\mathcal{P}_{n+1}).$
Let $A,B\in\mathcal{B}(\mathcal{P}_n).$ We prove this by cases for each step of the algorithm, confirming that the condition $A \subseteq B\cup \mbox{Ext}(B)$ in Lemma \ref{lem:lasvergnas_def} is satisfied. Denote $\hat{A}=\gamma(A)$ for any $A\in\mathcal{B}(\mathcal{P}_n).$ 
For the following cases, $\hat{A}$ denotes the basis produced by $\gamma$ after performing only the first step of the algorithm. Since every basis in $\mathcal{B}(\mathcal{P}_n)$ contains either 1 or 2, we need consider only the following three cases.

\textbf{Case 1:} $1\in A$ and $1\in B$.\\
Since $1\in A$, the only circuits contained in $A\cup\{i\}$ for $i\in[2n]$ are those of the form $(i,j)$ for $i,j\in[n+1,\dots,2n]$. Thus $\mbox{Ext}(A)=\{n+1,\dots,max(A)-1\}$. Similarly, $\mbox{Ext}(B)=\{n+1,\dots,max(B)-1\}$. After applying step one of $\gamma$, $1\in\hat{A}$ and $1\in\hat{B}$. Thus $\mbox{Ext}(\hat{A})=\{n+2,\dots,max(\hat{A})-1\}$ and $\mbox{Ext}(\hat{B})=\{n+2,\dots,max(\hat{B})-1\}$. Let $a\in \hat{A}\backslash \{1\}$. Then $a = a_i+1$ for $a_i \in A$. Since  $A\leq_{\mbox{Ext}}B$, it follows that $a_i \in B \cup \{n+1,\dots,max(B)-1\}$. Thus, $a_i+1\in B+1\cup \{n+2,\dots,max(B)\}$. Note that $max(\hat{B})-1 = max(B)$. Therefore, $a\in \hat{B}\cup Ext(\hat{B})$.

\textbf{Case 2:} $1\in A$ and $1\notin B$.\\ 
By definition, if $1\notin B$ then $1\in \mbox{Ext}(B)$ for every $B$ of a trivial UIP. Let $a\in \hat{A}\backslash \{1\}$. Then $a = a_i+1$ for $a_i \in A$. We know $a_i \in B \cup \{1,n+1,\dots,max(B)-1\}$. Thus, $a_i+1\in B+1\cup \{1,n+2,\dots,max(B)\}$. Note that $max(\hat{B})-1 = max(B)$. Therefore, $a\in \hat{B}\cup Ext(\hat{B})$.

\textbf{Case 3:} $1 \notin A \mbox{ or } B$.\\
By assumption we know that $1\in \mbox{Ext}(A)$ and $1\in \mbox{Ext}(B)$. Then for $a\in \hat{A}\backslash \{2\}$, we have that $a = a_i+1$ for $a_i \in A$. Since $a_i \in B \cup \{1,n+1,\dots,max(B)-1\}$, it follows $a_i+1\in B+1\cup \{1,n+2,\dots,max(B)\}$. Note that $max(\hat{B})-1 = max(B)$. Therefore, $a\in \hat{B}\cup Ext(\hat{B})$.

For the second step of the algorithm, notice that the only change to a basis $B$ is that the element $2n+1$ is replaced with $2(n+1)$. This means that $\mbox{Ext}(\hat{B})=\mbox{Ext}(B)\cup\{2n+1\}$. Then the cases to check are exactly those done for step one of the algorithm, and proceed in the same way. Similarly, for the third step, $\mbox{Ext}(\hat{B})=\mbox{Ext}(B)\backslash\{1\}$ since this step replaces the element 2 with 1. And then the cases proceed as above.

In all cases, the external order of $\mathcal{B}(\mathcal{P}_{n+1})$ is shown to be induced by the external order of $\mathcal{B}(\mathcal{P}_n)$. Thus, $\gamma(\mathcal{B}(\mathcal{P}_n))=\mbox{Ex}(\mathcal{P}_{n+1})$ as desired.
\end{proof}

\section{Conclusion and Future Work}\label{}

The recursive algorithm for the external poset on the bases of trivial UIPs provides a stepping stone to explore and develop either general or specific algorithms for the other UIPs. Recall that these are Catalan objects and that we have Catalan many to consider. In Figure \ref{fig:all rank 3} we see the trivial UIP and the remaining UIPs of rank $r=3$. There appears to be some symmetry in the Hasse diagrams, but it is unclear how these five UIPs give rise to the non-trivial UIPs of rank $r=4$. See Figure \ref{fig:soup} for examples of externally ordered posets for non-trivial UIPs of rank $r=4$. As was done for the recursion for the externally ordered poset for trivial UIPs, a potential first step is to find nice descriptions for the bases and circuits of the remaining UIPs.

Another approach for describing Catalan recursion for the remaining externally ordered posets is to consider matroid minors. One can check if minors of rank $n+1$ UIPs are isomorphic to rank $n$ UIPs. We have begun to explore this direction and our code is provided below. 

 Our algorithm provides a new way of producing externally ordered posets. It would be interesting to know how the complexity of Las Vergnas' procedure compares to ours. In particular, how efficient is our algorithm in generating the next rank trivial UIP.

\section{Acknowledgements}\label{}
This project would not be possible without the support and encouragement of various people I've been fortunate enough to meet and work with at UC Davis. I would like to thank Dr. Anastasia Chavez for her patience and guidance throughout this project. Thank you for giving me the opportunity to explore mathematics research, and for introducing me to matroid theory. I am also grateful to Dr. Jesus De Loera for supporting this project.
I am grateful to the McNair scholars program for introducing me to undergraduate research, Esha Datta for so many helpful conversations and introductions, Shannon Allen for listening to me talk about this project many, many times, Calvin Cramer for being my Python help desk, and the friends, family, faculty and staff that saw this project in various stages of completion. 

\begin{figure}[h]\label{fig:uip r=3}
    \centering
    \includegraphics[width=1\linewidth]{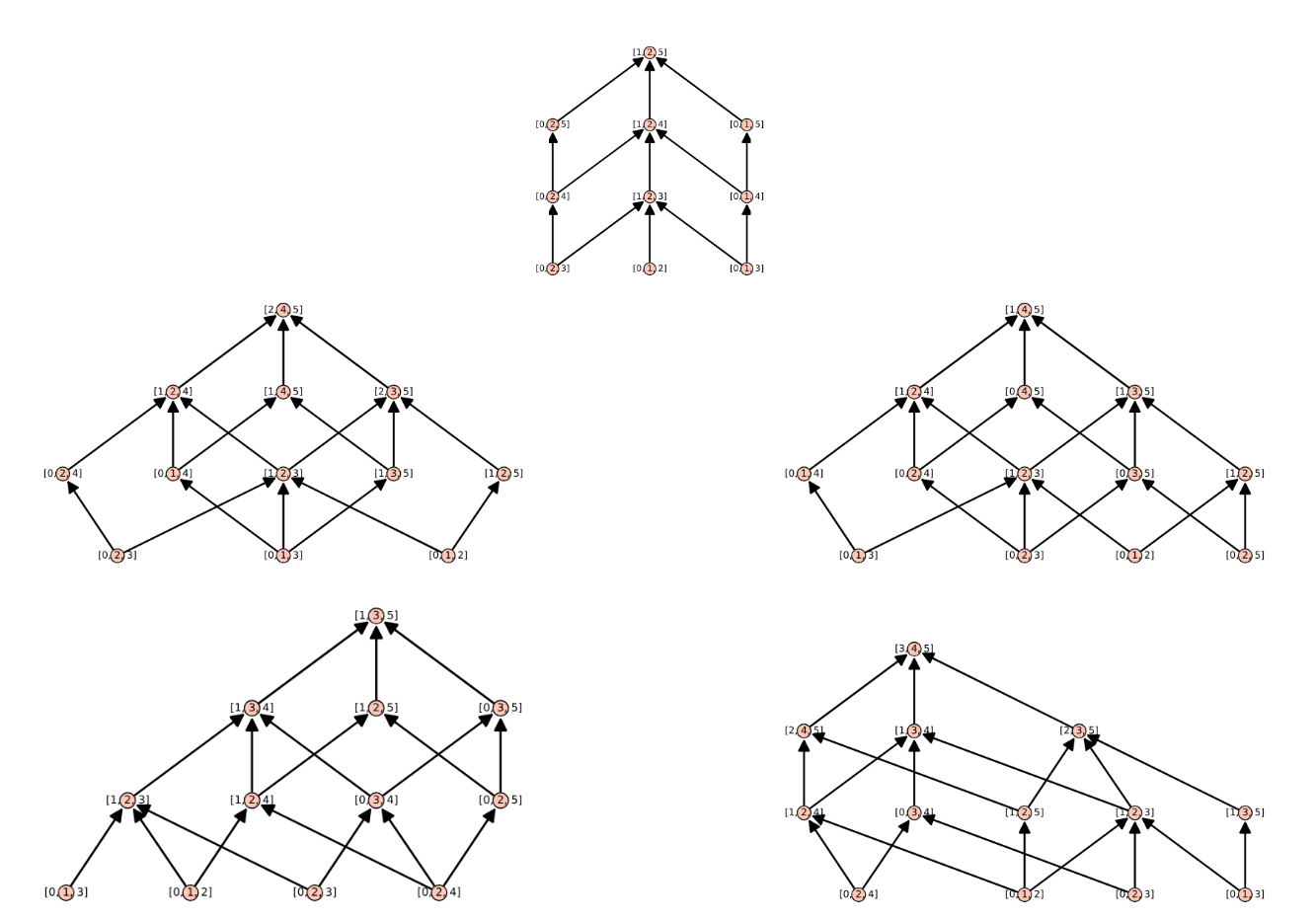}
    \caption{All UIPs with rank $r=3$.}
    \label{fig:all rank 3}
\end{figure}

\begin{figure}[h]%\label{fig:soup}
    \centering
    \includegraphics[width=1\linewidth]{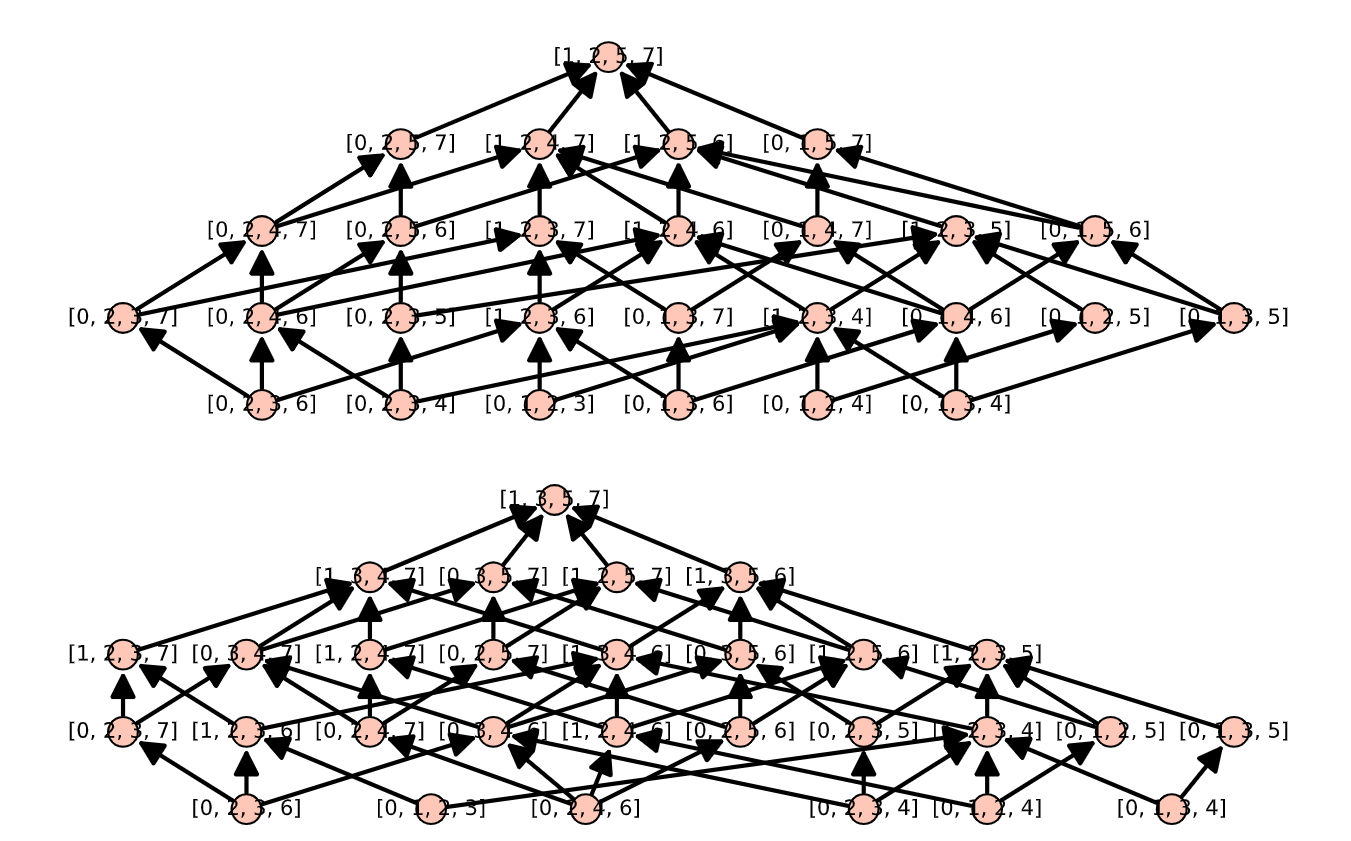}
    \caption{Two of the 14 UIPs with rank $r=4$.}
    \label{fig:soup}
\end{figure}
\pagebreak
\begin{lstlisting}[language=Python]
import sage.matroids.matroid
def UnitIntervalPositroid(Dyck_matrix):

    D = Dyck_matrix
    rank = D.nrows()
    D_flip = matrix(rank)
    for i in range(rank):
        D_flip[rank-i-1] = (-1)^i*D[i]
    A = matrix.identity(rank).augment(D_flip)
    cols = A.ncols()
    M_A = Matroid(A)
    
    return M_A

def powerset(fullset):
  listrep = list(fullset)
  n = len(listrep)
  return [[listrep[k] for k in range(n) if i&1<<k] for i in range(2**n)]
  
Pos_list3 = []
for d in D_matrices:
    U = UnitIntervalPositroid(d)
    Pos_list3.append(U)
print Pos_list3

E = (Matrix([[1,1,1,0],[1,1,1,0],[1,1,1,0],[1,1,1,1]]))
    
U4 = UnitIntervalPositroid(E)

contU4 = []
for s in [0,1,2,3,4,5]:
    for t in [0,1,2,3,4,5]:
        if t==s: break
        else:
            V = U4.delete(t).contract(s)
            contU4.append([V])

for p in contU4:
    for m in Pos_list3:
        if p.is_isomorphic(m)==True:
            print [b for b in m.bases()],"yes","\n"
            print 'bases of cont-del positroid are ', [b for b in p.bases()],"\n"

\end{lstlisting} 

\bibliographystyle{abbrv}
\bibliography{SeniorThesis}
\label{sec:biblio}

%\clearpage

%\input{sec-tutte-decomposition}
%\input{sec-algo-growing}

\end{document}